\documentclass[11pt]{amsart}
\usepackage{amsfonts, amssymb, amsmath, fancyhdr, url, amsthm}

\renewcommand{\leftmark}{{\scriptsize The spectral sequence of the canonical foliation on Vaisman manifolds}}

\newcommand{\al}{\alpha}

\newcommand{\ie}{{\em i.e. }}
\newcommand{\eg}{{\em e.g. }}

\newcommand{\ZZ}{\mathbb{Z}}

\def\eqref#1{(\ref{#1})}

\newcommand{\R}{{\mathbb R}}
\newcommand{\Q}{{\mathbb Q}}

\def\1{\sqrt{-1}\:}
\newcommand{\restrict}[1]{{\left|_{{\phantom{|}\!\!}_{#1}}\right.}}
\newcommand{\cntrct}                
{\hspace{2pt}\raisebox{1pt}{\text{$\lrcorner$}}\hspace{2pt}}

\newcommand{\cac}{\mathcal{C}}
\newcommand{\caf}{\mathcal{F}}

\renewcommand{\bar}{\overline}
\renewcommand{\phi}{\varphi}
\renewcommand{\epsilon}{\varepsilon}
\renewcommand{\geq}{\geqslant}

\newcommand{\SL}{\operatorname{SL}}

\numberwithin{equation}{section}

\makeatletter

\@addtoreset{equation}{section}
\@addtoreset{footnote}{section}
\makeatother

\newcounter{Mycounter}[section]
\newcounter{lemma}[section]
\newcounter{claim}[section]

\newcounter{sublemma}[section]

\newcounter{corollary}[section]

\newcounter{theorem}[section]

\newcounter{conjecture}[section]

\newcounter{proposition}[section]

\newcounter{definition}[section]

\newcounter{example}[section]

\newcounter{remark}[section]

\newcounter{problem}[section]

\newcounter{question}[section]

\begin{document}

\begin{center}
{\LARGE\bf
The spectral sequence of the canonical foliation of a Vaisman manifold
}\\[3mm]
{\large\bf
Liviu Ornea, and
Vladimir Slesar
}
\end{center}

\hfill

{\small
\hspace{0.15\linewidth}
\begin{minipage}[t]{0.7\linewidth}
{\bf Abstract} \\
In this paper we investigate the spectral sequence associated to a
Riemannian foliation which arises naturally on a Vaisman manifold. Using the
Betti numbers of the underlying manifold we establish a lower bound for the
dimension of some terms of this cohomological object. This way we obtain cohomological obstructions for two-dimensional foliations to be induced from a Vaisman structure. We show that if the foliation is quasi-regular the lower bound is realized.  In the final part
of the paper we discuss two examples.
\end{minipage}
}

\hspace{5mm}

\noindent{\bf Keywords:} locally conformally K\"ahler, canonical foliation, Vaisman manifold, spectral sequence.

\noindent {\bf 2000 Mathematics Subject Classification:} { 53C55.}\\[4mm]


\tableofcontents

\date{\today }

\section{Introduction}

A Hermitian manifold $(M,J,g)$ ($\dim_\R M\geq 2)$ is {\em Vaisman} if its fundamental two-form $\omega(X,Y):=g(X,JY)$ satisfies $d\omega=\theta\wedge\omega$ for a non-zero one-form $\theta$ (called the {\em Lee form}) which is parallel with respect to the Levi-Civita connection of the metric. In particular, a Vaisman metric is locally conformally K\"ahler (LCK), see \eg \cite{Dra-Orn}.

A Vaisman metric is a Gauduchon metric ($\partial\bar\partial\omega^{n-1}=0$) and hence, on compact $M$, it is unique, up to homothety, in a given conformal class. But not every conformal class of LCK metrics contains a Vaisman one: \eg the Inoue surfaces do not admit Vaisman metrics, \cite{Bel}.

Most of the compact complex surfaces are LCK, \cite{Br}, and among them, many  are Vaisman: Hopf surfaces of rank 1, Kodaira surfaces etc., see \cite{Bel}. Higher dimensional examples of Vaisman manifolds are the {\em diagonal Hopf manifolds}, see \cite{or_shell}.

Throughout this paper we shall assume $\theta$ is not exact, that is, the manifold is not {\em globally conformally K\"ahler}. In particular,  $b_1(M)\ge 1$.

On a Vaisman manifold $(M,J,g)$, the vector fields metrically equivalent with $\theta$ and respectively $-J\theta$ (called {\em the anti-Lee form}) define a real 2-dimensional distribution $\caf$ which is integrable and generate a Riemannian, totally geosesic and holomorphic foliation \cite{Vai}, usually called {\em the canonical foliation}. It was widely studied, especially in the compact case, \eg in
\cite{Tsu}, where it is used to show  that a direct product of compact Vaisman manifolds
cannot carry LCK metrics.  Geometric properties of this foliation
were discussed in \cite{Che-Pic}, while in \cite{Par} (see also the last
section of this paper)  a comprehensive description of
the canonical foliation is presented for some Hopf surfaces, addressing the transverse structure of
the foliation, a classification of leaves and in some particular cases even
a characterization of the space of leaves.

The purpose of this paper is to investigate \emph{the spectral sequence} $(E_k,d_k)$
of the canonical foliation of a compact Vaisman manifold (for definition see \cite{Al-Ko, Ton}).
The dimensions of the spectral terms $E_k$ with $k\ge 2$ are known to be invariant with respect to a leafwise
homeomorphism (\ie a homeomorphism which sends leaves to leaves), \cite{Al-Mas}. As in the
bundle case, the spectral sequence converges to the de\thinspace Rham cohomology of $M$.

For other spectral sequences defined on a foliated manifolds we refer to \cite{Vai2} for symplectic
foliations and more recently to \cite{Wo} for a spectral sequence related to the transverse geometry. Also, for the particular
case when the transverse distribution is integrable, see \cite{Pop}.

Note that the page  $E_2$ (which
is the most relevant part of the spectral sequence) contains as a subset
the groups of the \emph{basic cohomology} (the spectral terms with 0
leafwise degree). The basic de\thinspace Rham complex has been intensively studied and
in fact can be regarded as the counterpart of the de\thinspace Rham complex for the
space of leaves.

The transverse geometry of a foliation is encoded in its basic cohomology (for closed K\"ahler
manifolds see \eg \cite{Go}, for Sasakian and 3-Sasakian manifolds see \cite{bg}, \cite{Boy-Gal} etc.). The
 formulae which relate the {\em basic Betti numbers} with respect to the canonical foliation of a Vaisman manifold to the usual
Betti numbers of the manifold where given in \cite{Vai}. Moreover, using the minimality of the canonical foliation and the Poincar\'e duality defined in \cite{Mas} it is possible to compute the
dimension of the spectral terms with top leafwise degree 2.

In this paper we go further and compute the dimension of the remaining terms of order 2 in terms of the {basic Betti numbers}. We rely to the  fact that the remaining terms are related to the Lee and anti-Lee  one-forms which  are defined in the leafwise direction.
Together, these two steps give informations about all relevant spectral terms.

Our main tool is  a \emph{Hodge theoretic approach}, devised in this particular setting by
\'Alvarez-L\'opez and Kordyukov \cite{Al-Ko}. We stress that this approach is different from the previous
cohomological method used for computing the spectral sequence (see \emph{e.g.} \cite{Dom}).

On compact Vaisman manifolds, some spectral terms  are
nontrivially related to the basic cohomology groups and  we obtain a lower bound
for their dimension in terms of basic Betti numbers. For quasi-regular foliations (\ie  foliations
with all leaves compact) we show that the lower bound is attained and for a
class of examples we determine the dimension of all spectral terms $E_2$. Note that a Vaisman structure can always be deformed to one with quasi-regular canonical foliation, see \cite{or-ve-imm}.

Our result ca be viewed as a cohomological obstruction for a given foliated structure on a complex manifold to be determined by
some Vaisman structure.

The paper is organized as follows. In Section 2 we present the basic
properties of the Vaisman manifolds and of the spectral sequence associated
to a Riemannian foliation. In Section \ref{section III}, we establish a lower bound for the
dimension of some spectral terms for the canonical foliation of a compact Vaisman manifold. In Section \ref{section IV} we
prove that the lower bound is attained if the foliation is quasi-regular.
In Section 5 we present two examples: one (the diagonal Hopf manifold) which satisfies the obstructions, and another one on which the obstruction is effective.

\section{Preliminaries}

\subsection{Vaisman manifolds}

We  present  the needed back\-ground for Vaisman manifolds. For details, proofs and examples, please see \cite{Dra-Orn} and more recent papers by Ornea and Verbitsky.

Let $(M,J)$ be a complex manifold of real dimension $2n+2$, with $n\ge 1$ (tacitly assumed to be connected, of class $\cac^\infty$).
\begin{definition}
A {\em Vaisman} metric on $(M,J)$ is a Hermitian metric $g$ such that its fundamental two-form $\omega(X,Y):=g(X,JY)$ satisfies the integrability condition:
$$d\omega=\theta\wedge\omega,
$$
for some one-form $\theta$, called {\em Lee form}, which is parallel with respect to the Levi-Civita connection of $g$.

Then $g$ gives a {\em Vaisman structure} on $(M,J)$, and  $(M,J,g)$ is called a {\em Vaisman manifold}.
\end{definition}

\begin{remark}
A Vaisman manifold is locally conformally K\"ahler (LCK), as $\nabla\theta=0$ implies $d\theta=0$.
\end{remark}

\hfill

As $\theta$ is parallel, it can be considered of norm 1. Let $\theta^c:=-J\theta$ be the {\em anti-Lee form} ($\theta^c(X)=\theta(JX)$). The vector fields metrically equivalent with $\theta$ and $\theta^c$ will be denoted by $U$ and $V$, and they are unitary:
$$\theta(U)=1,\quad \theta^c(V)=1,\quad \theta(V)=0,\quad \theta^c(U)=0.$$

As the universal Riemannian cover of a Vaisman manifold is a metric cone of a Sasaki manifold (see \eg \cite{Orn-Ver}),
the local structure of a Vaisman manifold implies the existence of a local
Sasaki structure transverse to the flow generated by $U$. Thus we may  assume
the existence of a local orthonormal frame $\{f_{1},\ldots,f_n,f_{n+1,}\ldots,f_{2n},U,V=\xi \}$, such that
$\xi$ is the Reeb vector field and
\[
J(e_i)=e_{n+i}, \qquad
J(e_{n+i})=-e_i.
\]
The dual frame will be denoted $\{f_{1,}^{\flat },\ldots,f_n^{\flat},f_{n+1,}^{\flat}, \ldots,f_{2n}^{\flat },\theta, \theta^c\dot \}$.

The parallelism of $\theta$ immediately implies:
\begin{lemma} (\cite{Vai}) \label{connect_Vaisman}
The following equations hold on a Vaisman manifold:
\begin{equation}
\begin{split} \label{conex_Vaisman}
\nabla _{e_i}U&=0,  \\
\nabla _{e_i}V&=J(e_i),  \\
\nabla _UU&=\nabla _VU=0,   \\
\nabla _UV&=\nabla _VV=0.
\end{split}
\end{equation}
\end{lemma}
In particular,   $U$ and $V$ generate a foliation $\caf$ with leafwise dimension $2$ (called the \emph{canonical foliation} \cite{Che-Pic,Par,Tsu}), the leaves of which are minimal submanifolds.

 Moreover, it can be shown, \cite{Vai},
that $g$ is a \emph{bundle-like} metric, which means the foliation is locally a Riemannian submersion. Such spaces are called \emph{Riemannian foliations} \cite{Ton, Mo}.

The metric $g$ induces a splitting of the tangent bundle of the
underlying manifold $M$,
\begin{equation}
TM=Q\oplus T\mathcal{F},  \label{desc_sp_tang}
\end{equation}
with $T\mathcal{F}$ being the leafwise tangent bundle and $Q$ the orthogonal complement. If $\pi_Q$, $\pi_{T\mathcal{F}}$
are the canonical projection operators and
\[
\nabla^Q:=\pi_Q \circ \nabla, \quad\nabla^{T\mathcal{F}}:=\pi_{T\mathcal{F}} \circ \nabla,
\]
then the Levi-Civita connection also splits
\[
\nabla =\nabla ^Q+\nabla ^{T\mathcal{F}}.
\]

The following result collects several useful computational facts:

\begin{lemma}
The relations below hold on a Vaisman manifold:
\begin{equation}
\begin{split} \label{conex_forme_Vaisman}
\nabla _{e_i}^{T \mathcal{F}}U^{\flat }&=0,   \\
\nabla _{e_i}^{T \mathcal{F}}V^{\flat }&=0,   \\
\nabla _U^{T\mathcal{F}}e_i^{\flat }&=\nabla _V^{T\mathcal{F}}e_i^{\flat }=0, \\
\nabla _UU^{\flat }&=\nabla _VU^{\flat }=0,  \\
\nabla _UV^{\flat }&=\nabla _VV^{\flat }=0.
\end{split}
\end{equation}
\end{lemma}

\begin{proof}
The first two, and the last two relations are direct consequences  of the first equation
\eqref{conex_Vaisman}.  For the remaining one, just note  that
\[
\nabla _Ue_i^{\flat }(U) =-e_i^{\flat }(\nabla _UU)=0.
\]
\end{proof}

\subsection{The spectral sequence associated to a Riemannian foliation}

We now introduce  the spectral sequence of a Riemannian foliation. It  encodes many properties concerning the homotopy of the foliation, \cite{Al-Ko, Ton}.

Let $M$ be a closed (\ie compact and without boundary) manifold, and let $\caf$ be a foliation of dimension $p$ and codimension $q$ (hence   $n=p+q$). Let
\begin{equation}
\Omega _k^r=\left\{ \omega \in \Omega ^r\mid \iota_X\omega =0, X=X_1\wedge \cdots\wedge X_{r-k+1},\,X_i\in \Gamma \left(T\mathcal{F}\right) \right\},  \label{def filtrare}
\end{equation}
where $\iota$ denotes the interior product. This way, the set of differential forms $\Omega $ on $M$ becomes a filtered
complex:
\[
\Omega =\Omega _0\supset \Omega _1\supset ...\supset \Omega _q\supset \Omega_{q+1}=0.
\]

We follow \cite[Sec. 2]{Al-Ko}, \cite{Ton} for the definition of the spectral sequence associated to $\caf$.  Set  $E_0^{u,v}:=\Omega_u^{u+v}/\Omega _{u+1}^{u+v}$,
for $0\le u\le q$, $0\le v\le p$, and
construct the ``page'' of order $k+1$  inductively as $E_{k+1}:=H(E_k,d_k)$,
where $d_k$ is canonically induced by $d$. More precisely:
\[
E_{k+1}^{u,v}=\frac{\operatorname{Ker} (d_k:E_k^{u-k,v+k-1}\longrightarrow E_k^{u,v})}
{\operatorname{Im} (d_k:E_k^{u,v}\longrightarrow E_k^{u+k,v-k+1})}.
\]

An explicit description of the above terms is possible:

\[
E_k^{u,v}=\frac{Z_k^{u,v}}{Z_{k-1}^{u+1,v-1}+B_{k-1}^{u,v}},
\]
where the spaces $Z_k^{u,v}$, $B_k^{u,v}$ are:
\begin{equation}\label{B_Z}
\begin{split}
Z_k^{u,v} &:=\Omega _u^{u+v}\cap d^{-1}(\Omega _{u+k}^{u+v+1}), \\
B_k^{u,v} &:=\Omega _u^{u+v}\cap d(\Omega _{u-k}^{u+v-1}).
\end{split}
\end{equation}

\begin{definition}
The family of cohomological complexes $(E_k,d_k)_{k\ge 0}$ is called the spectral sequence of the foliation $\mathcal{F}$.
\end{definition}

\begin{remark}
The spectral terms $\{E_1^{u,0}\}_{0\le u\le q}$ can be identified with the spaces of \emph{basic forms} (with respect to $\caf$):
\[
\Omega_b^u:= \left\{ \alpha \in \Omega^u \mid \iota_{X}\alpha=0,\,
\iota_{X}d\alpha=0\, (\forall)\, X \in \Gamma \left(T\mathcal{F}\right) \right\},
\]
while $\{E_2^{u,0}\}_{0\le u\le q}$ can be identified with the groups $\{H_b^u\}_{0\le u\le q}$ of the
\emph{basic de\thinspace Rham cohomology} of the foliation.
\end{remark}

\hfill

The following result can be seen as an extension of the topological invariance
of the groups of basic cohomology stated in \cite{Kac-Nic1}:

\begin{theorem}
(\cite{Al-Mas})  On a compact Riemannian foliation the dimension of the spectral terms
$E_k^{u,v}$ is invariant with respect to foliated homeomophisms (\emph{i.e.}
homeomorphisms which send leaves to leaves) for $k\ge 2$, $0\le
u\le q$, $0\le v\le p$.
\end{theorem}

\hfill

Let now $(M,J,g)$ be a closed Vaisman manifolds of
dimension $2n+2$ and let $\caf$ be the canonical foliation. We denote by $\{b_i\}_{0\le i\le 2n+2}$
the Betti numbers of $M$ and by $\{e_i\}_{0\le i\le 2n}$ the basic Betti numbers. Then we have:

\begin{theorem}
\cite[Theorem 4.2]{Vai}\label{basic} On a compact Vaisman manifold the numbers $\{e_i\}$ are uniquely
determined by $\{b_i\}$ and, conversely, starting with Betti numbers $\{b_i\}
$ we can compute the basic Betti numbers $\{e_i\}$.
\end{theorem}

\hfill

As a consequence, if we describe the dimension of all spectral terms of
order 2 using the numbers $\{e_i\}$, then it would be possible to express them
using only  the de\thinspace Rham complex of $M$. We have
$$\dim E_2^{u,0}=\dim H_b^u:=e_u.$$
On the other hand, using the  Poincar\'e duality for the basic de\thinspace Rham complex of a minimal foliation
(see \emph{e.g.} \cite[Chapter 7]{Ton}), the duality stated in \cite{Mas} gives
$$\dim E_2^{u,0}=\dim E_2^{u,2}.$$
Then, the difficulty consists in  studying the terms $E_2^{u,1}$, for $0\le u\le 2n$.
To overcome it,  we use a Hodge-type theory  for the terms of the
spectral sequence, \cite{Al-Ko}, that we further describe.

 \subsubsection{\bf A Hodge-type theory} Let $\caf$ be a closed Riemannian foliation. The splitting \eqref{desc_sp_tang} induces a corresponding bigrading of $\Omega$ such that

\[
\Omega _k=\bigoplus\limits_{u\geq k}\Omega ^{u,\cdot }.
\]
Therefore \eqref{def filtrare} can be regarded as the set of
differential forms with transverse degree at least $k$. Then:
\begin{equation}
\label{splitting}
\Omega ^{u,v}=\cac^\infty \left(\bigwedge^uQ^{*}\oplus \bigwedge^v T\mathcal{F}^{*}\right),\,u,v\in \mathbb{Z}.
\end{equation}

Let $\pi _{u,v}:\Omega \rightarrow \Omega^{u,v}$ be the canonical projections
determined by the bigrading of $\Omega $. Define the topological
vector spaces:
\[
z_k^{u,v}=\pi _{u,v}(Z_k^{u,v}),\quad b_k^{u,v}=\pi _{u,v}(B_k^{u,v}),\quad
e_k^{u,v}=z_k^{u,v}/b_{k-1}^{u,v}.
\]

One can see that
\[
Z_k^{u,v}\cap \operatorname{Ker} \pi _{u,v}=Z_{k-1}^{u+1,v-1},
\]
and this induces the continuous linear isomorphism
\begin{equation}
E_k^{u,v}\simeq  e_k^{u,v}.  \label{e_E}
\end{equation}

Let now $d_k$ be the operator introduced above. In \cite[Section 5.1]{Al-Ko} a sequence
of Laplace type operators is inductively constructed: $\Delta _0,\Delta_1,\ldots$,
as well as a sequence of corresponding kernel spaces
$\mathcal{H}_1\supseteq \mathcal{H}_2\supseteq \cdots$, such
that:
\[
\Omega =\mathcal{H}_1\oplus \overline{\operatorname{Im}\,d_0}\oplus \overline{\operatorname{Ker}
\delta _0},
\]
\[
\mathcal{H}_1=\mathcal{H}_2\oplus \operatorname{Im}\,d_1\oplus \operatorname{Ker} \delta _1,
\]
\[
\vdots
\]
where the overline denotes closure with respect to $\cac ^\infty$ topology. These decompositions are the analogues of the Hodge
theory in our setting. If $\mathcal{H}_k^{u,v}:=\pi _{u,v}(\mathcal{H}_k)$
then these vector spaces are related by the following isomorphism
\cite[Section 5.1]{Al-Ko} (for $k=2$ see also \cite[ Theorem 2.2(iv)]
{Al-Ko}):

\[
e_k^{u,v}\simeq \mathcal{H}_k^{u,v}
\]
so together with \eqref{e_E} we obtain
\begin{equation}
\mathcal{H}_k^{u,v}\simeq E_k^{u,v},\quad \text{for}\, k\ge 2.  \label{E_H}
\end{equation}

The  de\thinspace Rham differential  and codifferential decompose with respect to the bigrading \eqref{splitting} as (\cite{Ton}):
\begin{equation}
d=d_{0,1}+d_{1,0}+d_{2,-1},\ \delta =\delta _{0,-1}+\delta _{1,0}+\delta_{-2,1},  \label{rost}
\end{equation}
with
\[
d_{i,j}:\Omega^{u,v} \rightarrow \Omega^{u+i,v+j}.
\]
Note that $\delta_{i,j}$ is the adjoint operator of $d_{i,j}$. Also, $d_0 \equiv d_{0,1}$.

\begin{remark}\label{mean curv e 0}
We can define the basic de\thinspace Rham operators $d_b$, $\delta_b$ restricting
$d_{1,0}$ and $\delta_{-1,0}$ to $\Omega_b$. If the foliation has vanishing mean curvature (for instance in the case of canonical foliations on Vaisman manifolds), then these operators coincide with the usual de\thinspace Rham operators on some local transverse submanifold. Consequently, the \emph{basic Laplace operator} $\Delta_b:=d_b \delta_b+\delta_b d_b$ coincides  with the corresponding transverse operator (see \eg \cite{Ton}). Similarly, using the 1-st order operators $d_{0,1}$, $\delta_{0,-1}$, we construct $\Delta_0:=d_{0,1}\delta_{0,-1}+\delta_{0,-1}d_{0,1}$. The \emph{leafwise Laplace operator} $\Delta_{\mathcal{F}}$ can be constructed using the restrictions $d_{\mathcal{F}}$, $\delta_{\mathcal{F}}$ of the first order operators $d_{0,1}$, $\delta _{0,-1}$ to $\Omega^{0,\cdot}$. Finally, we notice that $d_{0,1}$, $\delta _{0,-1}$ vanish on basic forms.
\end{remark}

\hfill

We investigate the kernel space $\mathcal{H}_2$ using the \emph{adiabatic limit} of the foliation.

The metric tensor can be written according to \eqref{desc_sp_tang}:
\[
g=g_Q \oplus g_{T\mathcal{F}}.
\]
Introducing a parameter $h>0$, we define the
family of metrics:
\begin{equation}
g_h=h^{-2}g_Q \oplus g_{T\mathcal{F}}.  \nonumber
\end{equation}
The pair represented by the manifold $M$ and the ``limit'' of the Riemannian manifolds $(M,g_h)$
when $h\downarrow 0$ is called the ``adiabatic limit'' of the
foliation $(M,\mathcal{F})$. This concept was introduced for the
first time by Witten, being a necessary tool for the study of the ``eta''
invariant of the Dirac operator; it was also used and extended by \cite{Bi-F} and \cite{Maz-Me}.

Let $\Theta _h:\left( \Omega ,g_h\right)
\rightarrow \left( \Omega ,g\right)$ be (\cite{Maz-Me}):
\begin{equation*}
\Theta _h\omega =h^u\omega ,\ \forall \omega \in \Omega ^{u,v}, \quad u,v\in\mathbb{N}.
\end{equation*}
We define the differential and
codifferential obtained by the ``rescaling'' procedure:
\begin{equation}
\begin{split} \nonumber
d_h &=\Theta _hd\Theta _h^{-1}, \\
\delta_h &=\Theta _h\delta _{g_h}\Theta _h^{-1}.
\end{split}
\end{equation}
Then (\ref{rost}) implies:
\begin{equation}
\begin{split} \label{d_delta_h}
d_h &=d_{0,1}+hd_{1,0}+h^2d_{2,-1},   \\
\delta _h &=\delta _{0,-1}+h\delta _{-1,0}+h^2\delta _{-2,1}.
\end{split}
\end{equation}
The corresponding Laplace and Dirac operators are:
\begin{equation*}
\begin{split}
\Delta _h&:=\Theta _h\Delta _{g_h}\Theta _h^{-1}=d_h\delta _h+\delta _hd_h,\\
D_h&:=d_h+\delta _h.
\end{split}
\end{equation*}
It can be shown
that $D_h$ is (formally) self-adjoint and $D_h^2=\Delta _h$.

We then have:

\begin{theorem}\label{Thm Al-Ko}
\cite[Section 1]{Al-Ko} Let $\{\al_i\}$ be a sequence of
differential forms  in $\Omega ^r$, with $\left\| \alpha_i\right\| =1$,
and let $h_i$ be a sequence of real numbers such that $h_i\downarrow 0$. If
\[
\left\langle \Delta _{h_i}\alpha _i,\alpha _i\right\rangle \in o\left(
h_i^2\right) ,
\]
then there exists a subsequence of $\alpha _i$ which converges in $\mathcal{H}_2^r$.
\end{theorem}

\begin{theorem}\label{Thm noua}
\cite[Section 5.1]{Al-Ko} The spaces $\mathcal{H}_2^{u,v}$ are uniquely
determined by $z_2^{u,v}$, $b_1^{u,v}$ and $b_0^{u,v}$ as follows:
\[
z_2^{u,v}+\overline{ b_0^{u,v}}=\mathcal{H}_2^{u,v}\oplus (b_1^{u,v}+\overline{b_0^{u,v}}),
\]
where the closure of $b_0^{u,v}$ is considered in the $\cac ^ \infty$ topology.
\end{theorem}

\begin{remark}
In \cite{Al-Ko}, the above are proven for general $\mathcal{H}_k$, not only for $k=2$.
\end{remark}

\hfill

Two direct consequences will be of interest for us:

\begin{lemma}
\label{Lema 1}If $\alpha \in \Omega ^{u,v}$ verifies
\begin{eqnarray}
d_{0,1}\alpha =0, \quad \delta _{0,-1}\alpha =0,  \label{d_0} \\
d_{1,0}\alpha =0, \quad \delta _{-1,0}\alpha =0,  \nonumber
\end{eqnarray}
then $\alpha \in \mathcal{H}_2^{u,v}$.
\end{lemma}

\begin{proof} We fix $\alpha _i=\alpha $. Clearly we can assume, without
restricting the generality, that $\left\| \alpha \right\| =1$. Then, for any $h_i\downarrow 0$, from \eqref{d_delta_h} and \eqref{d_0} we obtain:
\begin{equation}
\begin{split} \nonumber
\left\langle \Delta _{h_i}\alpha ,\alpha \right\rangle =\left\|
D_{h_i}\alpha \right\| ^2
=h_i^4\left\|(d_{2,-1}+\delta _{-2,1})\alpha \right\| ^2 \in o\left( h_i^2\right).
\end{split}
\end{equation}
The result now follows from   \ref{Thm Al-Ko}.
\end{proof}

\begin{lemma}
\label{Lem 2}If $\alpha \in \mathcal{H}_2^{u,v}$, then there exist $\beta_i\in \Omega ^{u+1,v-1}$ and
$\gamma _i\in \Omega ^{u-1,v+1}$ such that

\begin{itemize}
\item[(i)] $d_{0,1}\alpha =0$, \quad $\delta _{0,-1}\alpha =0$,

\item[(ii)] $d_{1,0}\alpha +d_{0,1}\beta _i\longrightarrow 0$,\quad $\delta _{-1,0}\alpha
+\delta _{0,-1}\gamma _i\longrightarrow 0$ in the $L^2$ norm.
\end{itemize}

\end{lemma}

\begin{proof} (i) is just  \ref{mean curv e 0}, taking into account that
$\mathcal{H}_2\subset \mathcal{H}_1$, and hence  $\Delta _0\alpha =0$.

To prove (ii), note that $\alpha \in \mathcal{H}_2^{u,v}$ implies ( \emph{via}  \ref{Thm noua})
 $\alpha \in z_2^{u,v}+\overline{b_0^{u,v}}$.
Then $\alpha =\alpha ^{\prime }+\alpha ^{\prime \prime }$, with
$\alpha^{\prime }\in z_2^{u,v}$ and $\alpha ^{\prime \prime }\in \overline{b_0^{u,v}}$.
If $\alpha ^{\prime }\in z_2^{u,v}$, then from the definition \eqref{B_Z} of the spaces $Z_k^{u,v}$
there exists a differential form $\beta \in \Omega^{u+1,v-1}$ such that
\begin{equation}
d_{1,0}\alpha ^{\prime }+d_{0,1}\beta =0,  \label{conditie z_2}
\end{equation}
and thus $d_{1,0}\alpha ^{\prime }\in \overline{d_{0,1}}(\Omega ^{u+1,v-1})$.

As $d^2=0$, one has (\cite{Al-Ko}):
\begin{equation}
d_{1,0}d_{0,1}+d_{0,1}d_{1,0}=0.  \label{d_1_0 si d_0_1}
\end{equation}
Let  $\operatorname{P}$ be the projection of
$\Omega $ to $\overline{d_{0,1}}(\Omega)$. Then from \eqref{d_1_0 si d_0_1}
we derive  (\cite[Lemma 2.3]{Al-Ko}):
\[
d_{1,0}\operatorname{P}=\operatorname{P}d_{1,0}\operatorname{P},
\]
and hence $d_{1,0}\alpha ^{\prime \prime }\in \overline{d_{0,1}}(\Omega ^{u+1,v-1})$. Then
$d_{1,0}\alpha \in \overline{d_{0,1}}(\Omega ^{u+1,v-1})$, and clearly there is a sequence $\beta_i$ such that
\[
d_{1,0}\alpha +d_{0,1}\beta _i\longrightarrow 0
\]
in the $L^2$ norm.

For the last part we use the idea in \cite[Corollary 5.15]{Al-Ko}. If  $M$ is not orientable, take its two-sheeted oriented
cover. Denoting by $*$ the Hodge star operator, we
obtain the above result for $*\alpha $, with a corresponding sequence $\beta_i$. As
\begin{equation}
\begin{split} \nonumber
\ast d_{1,0} &=(-1)^{r+1}\delta _{-1,0}*, \\
\ast d_{0,1} &=(-1)^{r+1}\delta _{0,-1}*,
\end{split}
\end{equation}
when the above operators are applied to a differential form of degree $r$,
we obtain the necessary sequence $\gamma _i$ from $\beta _i$.
\end{proof}

\section{A lower bound for $E_2^{u,1}$\label{section III}}

In this section, $(M,J,g)$ is a closed Vaisman manifold.

We start with a technical result about the operators defined in formula  \eqref{rost}:

\begin{lemma}\label{Lema 3}On a closed Vaisman manifold the 1-st order differential operators
$d_{0,1}$, $d_{1,0}$, and their adjoints vanish on $\theta$ and $\theta^c$:
\begin{equation*}
\begin{split}
d_{1,0}\theta&=0, \quad \delta_{-1,0}\theta=0,\quad  d_{1,0}\theta^c=0,\quad  \delta_{-1,0}\theta^c=0,\\
d_{0,1}\theta&=0,\quad \delta_{0,-1}\theta=0,\quad d_{0,1}\theta^c=0,\quad \delta_{0,-1}\theta^c=0.
\end{split}
\end{equation*}
\end{lemma}

\begin{proof} From $d\theta =0$, we obtain
\[
d_{0,1}\theta =0, \quad d_{1,0}\theta =0.
\]
On the other hand, using the transverse complex coordinates $\{z^a\}$ and the metric coefficients
$g_{a\bar b}$ with respect to these coordinates, we have (\cite{Vai}):

\[
d\theta^{c} =-ig_{a\bar b}dz^a\wedge d\bar z^b,
\]
and hence
\begin{equation}
d_{0,1}\theta^{c} =0, \quad d_{1,0}\theta^{c} =0.  \nonumber
\end{equation}
Along the leaves of $\caf$, the de\thinspace Rham codifferential
coincide with the codifferential operator on the leaves (considered as
immersed submanifolds). Applying \eqref{conex_forme_Vaisman}, we get
\begin{equation}\label{delta_0_omega}
\delta _{0,-1}\theta =-\iota _U\nabla _U^{T\mathcal{F}}\theta
-\iota_V\nabla_V^{T\mathcal{F}}\theta =0.
\end{equation}
Similarly, $\delta _{0,-1}\theta^{c} =0$. As $\theta,\theta^{c} \in \Omega^{0,1}$, the result follows
considering the bigrading.
\end{proof}

\hfill

We now give the  lower bound estimate. Recall that $e_i$ are the basic Betti numbers with respect to $\caf$.

\begin{theorem}\label{Thm 1}If $(M,\mathcal{F})$ is the canonical foliation of a closed Vaisman
manifold, then:
\[
\dim E_2^{u,1}\ge 2\, e_u, \quad \text{for} \quad 0\le u\le 2n.
\]
\end{theorem}

\begin{proof} We prove that if $\eta\in \Omega _b^u$ is a basic harmonic
differential form, then $\eta \wedge \theta $ and $\eta\wedge \theta^{c}
\in \mathcal{H}_2^{u,1}$.

 Using  \ref{mean curv e 0}, we get
\begin{equation}
\begin{split} \nonumber
d_{0,1}(\eta \wedge \theta ) &= (-1)^u \eta \wedge d_{0,1}\theta , \\
\delta _{0,-1}(\eta \wedge \theta ) &= (-1)^u \eta \wedge \delta _{0,-1}\theta .
\end{split}
\end{equation}

From \ref{Lema 3}  we then have
\[
d_{0,1}(\eta \wedge \theta )=0,\,\,\delta _{0,-1}(\eta \wedge \theta)=0.
\]

We show now that $d_{1,0}(\eta \wedge \theta )=0$ and
$\delta_{-1,0}(\eta \wedge \theta )=0$. Because $\eta $ is basic and
harmonic with respect to the basic Laplace operator $\Delta_b$, using  \ref{mean curv e 0}
we get
\[
d_{1,0}\eta =0, \quad \delta _{-1,0}\eta =0.
\]

From the hypothesis,  \ref{Lema 3} and \eqref{conex_forme_Vaisman} we  obtain
\begin{equation*}
d_{1,0}(\eta \wedge \theta ) =d_{1,0}\eta \wedge \theta +(-1)^u\eta
\wedge d_{1,0}\theta^{c} =0.
\end{equation*}

Then, again using \eqref{conex_forme_Vaisman}, we have:
\begin{equation}
\begin{split} \nonumber
\delta _{-1,0}(\eta \wedge \theta ) &=(-\sum_i\iota _{e_i}\nabla_{e_i}
-\iota _U\nabla _U-\iota _V\nabla _V)_{-1,0}\theta \\
&=\delta _{-1,0}\eta \wedge \theta +(-1)^u\eta \wedge \delta_{-1,0}\theta \\
& +\sum_i\iota _{e_i}\eta \wedge \nabla _{e_i}^{T\mathcal{F}}\theta
+(-1)^u\nabla _U^{T\mathcal{F}}\eta \wedge \iota _U\theta \\
& +(-1)^u\nabla _V^{T\mathcal{F}}\eta \wedge \iota _V\theta=0.
\end{split}
\end{equation}
The conclusion comes from  \ref{Lema 1} and equation \eqref{E_H}.

As for $\theta^c$, the proof is similar.
\end{proof}

\begin{corollary} Using \cite[Theorem 4.2]{Vai}, we can also express the lower bound using the Betti numbers of the underlying manifold:
\[
\dim E_2^{u,1}\ge 2\, (-1)^u\sum_{i=0}^{[u/2]}\left(\left[\frac{u}{2}\right]-i+1 \right)(b_{2i}-b_{2i-(-1)^u}), \quad \text{for} \quad 0\le u\le n.
\]
For $n+1\le u\le 2n$ we can use the Poincar\'e duality.
\end{corollary}

\begin{remark}
In particular, $\dim E_2^{0,1}\ge 2$.
\end{remark}

\hfill

This can be used to obtain the following obstruction for a 2-di\-men\-sio\-nal foliation on a compact complex manifold to be associated to a Vaisman structure:

\begin{proposition}
\label{criterii}Let $(M,J)$ be a closed, complex, foliated manifold with a $2$-dimensional foliation which admits
a bundle-like metric. If $e_{2n}=\dim E_2^{2n,0}\newline=0$ or $\dim E_2^{0,1}<2$, then the
foliation does not come from a Vaisman structure.
\end{proposition}

\begin{remark}
The top dimensional basic cohomology group $H_b^{2n}$ is related to the \emph{tautness} of the
foliation (see \emph{e.g.} \cite{Car,Mas}), but the geometrical meaning of the spectral term $E_2^{0,1}$
is still not understood. It seems that it could be related to the existence of a minimal sub-flow on the underlying manifold.
We notice here that on a Vaisman manifold there are two such flows, generated by the vector fields $A$ and $B$.
\end{remark}

\section{Quasi-regular foliations\label{section IV}}

In this section we show that if the canonical foliation has compact leaves, then  the inequalities in
 \ref{Thm 1} become equalities.

Recall that a foliated map is a pair
$(\mathcal{U},\varphi)$ such that $\varphi(\mathcal{U})\simeq \mathcal{O}_1\times \mathcal{O}_2$, with $\mathcal{O}_1\in \mathbb{R}^q$, $\mathcal{O}_2\in \mathbb{R}^p$, and $\varphi^{-1}(x_1 \times \mathcal{O}_2)$ is on the same leaf, for any $x_1\in \mathcal{O}_1$. If around any point $x$ there exists a foliated map with the property that any leaf $L$ intersects a transversal through $x$
at most a finite number of times $N(x)$, then the foliation is said to be \emph{quasi-regular}. If, furthermore
$N(x) = 1$ for any $p$, then it is \emph{regular}. The quasi-regularity is known to be equivalent to the compactness
of all leaves \cite{Boy-Gal}.

\begin{theorem}
\label{Thm 2} If the canonical foliation of a closed Vaisman manifold is
quasi-regular, then $\dim E_2^{u,1}=2\, e_u$.
\end{theorem}

\begin{proof} We prove
that if $\alpha \in \mathcal{H}^{u,1}$ then it is possible to decompose it as
\begin{equation}
\alpha =\eta _1\wedge \theta +\eta _2\wedge \theta^{c} , \label{alpha_from_delta}
\end{equation}
with $\eta _1,\eta _2$ basic harmonic forms with respect to the basic Laplace operator $\Delta_b$.

The proof is divided in two steps: first we obtain $\eta_i \in \Omega_b^u$, $1\le i \le 2$, then we show that
$\eta_i$ are also harmonic.

{\noindent \bf Step 1: $\eta_i$ are basic.} Let  $\alpha \in \mathcal{H}_2^{u,1}$ be fixed arbitrarily. From   \ref{Lem 2} (ii) we have
\begin{equation}
d_{0,1}\alpha =0,\qquad \delta _{0,-1}\alpha =0.  \label{d_delta_0_1_0}
\end{equation}
We then decompose $\alpha $ as
\[
\alpha =f_1\tilde\eta _1\wedge \theta+f_2\tilde\eta _2\wedge
\theta^c,
\]
with $\tilde\eta _i\in \Omega _b^u$ basic forms for $i=1,2$, and
$f_i\in \cac^\infty (M)$. As the foliation is quasi-regular, all leaves will be
compact, diffeomorphic with the real 2-torus $T^2$.

{\noindent \bf The proof for ${u=0}$.} We have
\[
\alpha =f_1\theta+f_2\theta^c.
\]

Note that the above functions $f_i$ are constant on the leaves of $\caf$, and hence they are basic. Indeed,  $d_{0,1}\alpha =0$, $\delta _{0,-1}\alpha =0$ imply $\Delta_{\mathcal{F}}\alpha=0$, as $\Delta_{\mathcal{F}}$ is the Laplace operator on the torus $T^2$ endowed with flat Euclidean metric. If we fix a leaf $L$,
then $\alpha _L:=\alpha _{\mid L}$ will be a harmonic 1-form and as $\theta\restrict{L}$,
$\theta^c\restrict{L}$ are also harmonic (as $\theta$ is parallel), then $f_1\restrict{L}$, $f_2\restrict{L}$ are constant, as
$\dim \operatorname{Ker}\Delta_{\mathcal{F}\mid L}=2$.

Moreover, we have
\begin{equation}
\begin{split} \label{eq alpha 1}
V(f_1) &=U(f_2),   \\
U(f_1) &=-V(f_2),
\end{split}
\end{equation}
which are equivalent with the conditions
\begin{equation}
\begin{split} \nonumber
d_{0,1}\alpha &=(\theta\wedge \nabla _U^{T\mathcal{F}}+\theta^c\wedge
\nabla _V^{T\mathcal{F}})\alpha =0, \\
\delta _{0,-1}\alpha &=(-\iota _U\nabla _U^{T\mathcal{F}}-\iota_V\nabla_V^{T\mathcal{F}})\alpha =0.
\end{split}
\end{equation}
{\noindent \bf The proof for $u\geq 1$.} Write now
\begin{equation}
\begin{split} \nonumber
d_{0,1}\alpha &=\tilde{\eta} _1\wedge V(f_1)\theta^c\wedge \theta+\tilde{\eta} _2\wedge U(f_2)\theta\wedge \theta^c \\
&=(U(f_2)\tilde{\eta} _2-V(f_1)\tilde{\eta} _1)\wedge \theta\wedge \theta^c,
\end{split}
\end{equation}
and
\[
\eta _{0,-1}\alpha =-U(f_1)\tilde{\eta} _1-V(f_2)\tilde{\eta} _2.
\]

We show that the differential forms  $f_i\tilde{\eta} _i$
are basic. As above, we fix a leaf $L$ and we suppose   ${\tilde{\eta}_i}\restrict{L}\neq 0$. As $\tilde{\eta} _i\in
\Omega _b^{u,0}$, from \eqref{d_delta_0_1_0} we have
\[
{\tilde{\eta}_2}\restrict{L}=a{\tilde{\eta}_1}\restrict{L},
\]
with $a\in \mathbb{R}$, and we obtain \eqref{eq alpha 1} for $f_1\restrict{L}$, $af_2\restrict{L}$.
Then $f_i$ are basic, as well as $\eta_i$.

{\noindent\bf Step 2: $\eta_i$ are harmonic with respect to the basic Laplacian.}
Let $\beta \in \Omega ^{u+1,0}$. We study the convergence
\[
d_{1,0}\alpha +d_{0,1}\beta _i\longrightarrow 0
\]
in $L^2$  around a regular point $x\in M$. As the subset of
regular points is open and dense (see \emph{e.g.} \cite[Chapter 3]{Mo}), taking a transversal $\mathcal{T}$ small enough
we can consider a foliated map
$(\mathcal{U},\varphi )$, $x\in \mathcal{U}$, such that $\mathcal{U}\simeq
\mathcal{T}\times T^2$, all leaves in $\mathcal{U}$ being regular, diffeomorphic to $T^2$.

We introduce transverse coordinates $y=(y^1,\ldots,y^{2n})$ and leafwise
coordinate $(t,s)$, such that $\varphi (\mathcal{U})=\mathcal{O}\times (0,2\pi )\times
(0,2\pi )$, with $\mathcal{O}\subseteq \mathbb{R}^{2n}$ and such that
\begin{equation}
U=c_1\frac \partial {\partial t},\qquad V=c_2\frac \partial {\partial s},
\label{U_V_dt_ds}
\end{equation}
with $c_1$, $c_2$ real constants.

As the leaves have trivial holonomy and the metric is bundle-like, sliding
along leaves we can construct on $\mathcal{U}$ a transverse orthonormal
basis $\{e_i\}_{1\le i\le q}$ such that $\{e_1^{\flat }\}\in \Omega _b^1$,
with $\{e_1^{\flat }\}$ basic forms. As above, for local computation we
consider the dual basis $\{e_1^{\flat },\ldots,e_{2n}^{\flat },\theta, \theta^c\}$.

Locally, we can write
\begin{equation}
\begin{split} \nonumber
\beta =g^{i_1\ldots i_{u+1}}e_{i_1}^{\flat }\wedge \cdots\wedge
e_{i_{u+1}}^{\flat }
=g^Ie_I^{\flat },
\end{split}
\end{equation}
for $1\le i_1,\ldots,i_{u+1}\le q$, and the multi-index $I:=(i_1,\ldots,i_{u+1})$,
with $e_I^{\flat }:=e_{i_1}^{\flat }\wedge \cdots\wedge e_{i_{u+1}}^{\flat }$.

Using \eqref{U_V_dt_ds}, we obtain
\begin{equation}
\begin{split} \label{d_1_0_beta_local}
d_{0,1}\beta &=c_1\frac{\partial g^I}{\partial t}\theta\wedge
e_I^{\flat }+c_2\frac{\partial g^I}{\partial s}\theta^c\wedge e_I^{\flat} \\
&=\frac{\partial {g}_1^I}{\partial t}e_I^{\flat }\wedge \theta+
\frac{\partial {g}_2^I}{\partial s}e_I^{\flat }\wedge \theta^c,
\end{split}
\end{equation}
with ${g}_i^I:=(-1)^{u+1}c_i\cdot g^I$, $i=1,2$.

From \eqref{alpha_from_delta} and \eqref{conex_forme_Vaisman} we get
\begin{equation}
d_{1,0}\alpha ={h}_1^I e_I^{\flat }\wedge \theta
+{h}_2^Ie_I^{\flat }\wedge \theta^c,  \label{d_1_0_alpha_local}
\end{equation}
As $d_{1,0}{\delta }_i$, $e_J^{\flat }$ are basic differential forms,
$h_i^I$ are basic functions, $i=1,2$.

In the sequel we check the $L^2$ convergence stated in  \ref{Lem 2} (ii) on the local chart $\mathcal{U}$.
From \eqref{d_1_0_beta_local} and \eqref{d_1_0_alpha_local}, we have
\begin{equation}
\begin{split} \nonumber
\mathcal{I} &:=\int\limits_{\mathcal{U}} \left\| d_{1,0}\alpha
+d_{0,1}\beta \right\| ^2d\mu _g \\
&=\sum_I\int \limits_{\varphi(\mathcal{U})}\left( {h}_1^I
+\frac{\partial {g}_1^I}{\partial t}\right)^2+\left( {h}_2^I
+\frac{\partial {g}_2^I}{\partial s}\right)^2\sqrt{\det G}\,ds\,dt\,dy,
\end{split}
\end{equation}
where $G$ is the matrix of the metric $g$ with respect to the local chart,
$d\mu _g$ is the volume form canonically associated to the metric $g$, $dy$ corresponding to the transverse coordinates.

We fix now $I$. Clearly the Riemannian metric $g$ is non-degenerate and $\mathcal{U}$ can be chosen relatively compact.
Then we can find a constant $c$ such that
\[
\sqrt{\det G}>c>0  \label{G_c}
\]
on $\mathcal{U}$. Then, as ${h}_1^I$ is basic (so it does not depend
on $t$ and $s$), using the Fubini formula, we have
\begin{equation}
\begin{split} \nonumber
\int \limits_{\varphi(\mathcal{U})}\left( {h}_1^I+\frac{\partial
{g}_1^I}{\partial t}\right) ^2\sqrt{\det G}\, ds\,dt\,dy &>c4\pi^2\int\limits_{\mathcal{O}}({h}_1^I)^2dy \\
&+2c\int\limits_{\mathcal{O}}{h}_1^I\int\limits_0^{2\pi }
\left(\int\limits_0^{2\pi }\frac{\partial {g}_1^I}{\partial t}dt\right)ds\,dy \\
&+c\int \limits_{\varphi(\mathcal{U})}\left( \frac{\partial {g}_1^I}
{\partial t}\right)^2ds\,dt\,dy.
\end{split}
\end{equation}

Because on the torus ${g}_1^I(0)={g}_1^I(2\pi )$, the second term vanishes
and the third is positive, then arguing in a similar manner, for arbitrary $I$, we can write

\[
\left\| d_{1,0}\alpha +d_{0,1}\beta \right\| _{L^2}^2 \ge \mathcal{I}
>4\pi^2c\sum_I\int\limits_\mathcal{O}(({h}_1^I)^2+({h}_2^I)^2)dy.
\]

As the last expression depends only on $\alpha $ and on $\mathcal{U}$, we can have
a sequence $\beta _i\in \Omega ^{u+1,0}$ such
that
\[
d_{1,0}\alpha +d_{0,1}\beta _i\longrightarrow 0\mbox{\, in \,}L^2
\]
if and only if ${h}_i^I=0$ on $\mathcal{U}$, for any $I$ and $1\le i\le 2$.
All mathematical objects that we use are of type $\cac^\infty $, $x$ is a
regular point arbitrarily chosen and the set of regular points is dense.
Consequently we obtain the desired relation
\begin{equation}
d_{1,0}\mathbf{\delta }_i=0  \label{d 0 1 delta}
\end{equation}
on $M$.

We still have to prove $\delta _{-1,0}\mathbf{\delta }_i=0$. We proceed in a
similar way as above.

Notice that
\begin{equation}
\begin{split} \nonumber
\delta _{-1,0}\alpha &=\delta _{-1,0}\eta_1\wedge \theta
+\delta _{-1,0}\eta_2\wedge \theta^c \\
&={h}_1^Ie_I^{\flat }\wedge \theta+{h}_2^Ie_I^{\flat}\wedge \theta^c,
\end{split}
\end{equation}
with $I:=(i_1,\ldots,i_{u-1})$, $1\le i_1,\ldots,i_{u-1}\le 2n$ and ${h}_i^I$
being basic functions.

We consider $\gamma \in \Omega ^{u-1,2}$; on $\mathcal{U}$ we write
\[
\gamma =g^Ie_I^{\flat }\wedge \theta\wedge \theta^c.
\]
Then
\begin{equation}
\begin{split} \nonumber
\delta _{0,-1}\gamma &=(-1)^{u-1}\bigl( c_1U(g^I)e_I^{\flat }\wedge
\theta^c-c_2V(g^I)e_I^{\flat }\wedge \theta)\bigr) \\
\ &=\frac{\partial {g}_1^I}{\partial s}e_I^{\flat }\wedge \theta
+\frac{\partial {g}_2^I}{\partial t}e_I^{\flat }\wedge \theta^c.
\end{split}
\end{equation}
and, consequently
\begin{equation}
\begin{split} \nonumber
\mathcal{J} &:=\ \int \limits_{\mathcal{U}}\left\| \delta _{-1,0}\alpha
+\delta_{0,-1}\gamma \right\| ^2d\mu _g \\
\ &=\sum_I\int \limits_{\varphi(\mathcal{U})}\left( {h}_1^I
+\frac{\partial {g}_1^I}{\partial s}\right) ^2+\left( {h}_2^I
+\frac{\partial {g}_2^I}{\partial t}\right) ^2\sqrt{G}\,ds\,dt\,dy.
\end{split}
\end{equation}

As before, we get
\[
\left\| \delta _{-1,0}\alpha +\delta _{0,-1}\gamma \right\|_{L^2}^2\ge
\mathcal{J}>0,
\]
and we obtain ${h}_i^I=0$ on $\mathcal{U}$ for any $I$. From here,
arguing as above,
\begin{equation}
\delta _{-1,0}\mathbf{\delta }_i=0.  \label{delta 1 0 delta}
\end{equation}

From \eqref{d 0 1 delta} and \eqref{delta 1 0 delta} it results that
$\mathbf{\delta }_1$, $\mathbf{\delta }_2$ are basic harmonic forms with
respect to the basic operator $\Delta _b$. Using \eqref{E_H}

\[
\dim E_2^{u,1}=\dim H_2^{u,1}=2\,  e_u,
\]
and the theorem is proved.
\end{proof}

\section{Examples}\label{section V}

In this final section we present two foliated manifolds. The first one is the
canonical foliation of a diagonal Hopf surface, for which we apply  \ref{Thm 1} and \ref{Thm 2} to explicitly compute  the dimension of all spectral terms of order 2. The second example is  a class of foliations of
arbitrary large transverse dimension which does not come from a Vaisman
structure, in accordance with  \ref{criterii}.

\subsection{The diagonal Hopf surface} \label{ex 1} Let  $W:=\mathbb{C}^2\backslash \{0\}$ and $a$, $b\in
\mathbb{C}\backslash \{0\}$ and let $\gamma :W\rightarrow W$ be given by
\[
\gamma :\left( z_1,z_2\right) \longmapsto \left( az_1,bz_2\right).
\]
The quotient
\[
H_{a,b}:=W/\gamma ,
\]
is a Hopf surface, \cite{Gau-Orn, Par}. For any $\left(
z_1,z_2\right) \in W$ there is an unique real number $\phi $ such that
\[
\left| z_1\right| ^2\left| a\right| ^{-2\phi }+\left| z_2\right| ^2\left|
b\right| ^{-2\phi }=1.
\]
The map $\tilde \psi $ defined by
\[
\tilde \psi :\left( z_1,z_2\right) \longmapsto (\phi \mbox{\,mod\,}\mathbb{Z},z_1a^{-\phi },z_2b^{-\phi })
\]
is compatible with $\gamma $ and the quotient map $\psi :=$ $\tilde \psi
/\gamma $ establishes a homeomorphism between $H_{\alpha ,\beta }$ and the
product of spheres $S^1\times S^3$, \cite{Gau-Orn}. The Hermitian
form
\[
\Omega :=-\1\frac 14\frac{\partial\bar\partial\Phi }\Phi,\,\, \text{with}\,\, \log\Phi :={(\log \left| a \right| +\log \left| b \right|)\phi },
\]
produces  a Vaisman metric $g$ on  $H_{\alpha,\beta }$,
with $g(\cdot ,\cdot ):=-\Omega (J\cdot ,\cdot )$.
The topological properties of the leaves of the canonical foliation are
investigated in \cite{Par}, where the author proves that all leaves are
compact (and consequently the foliation is quasi-regular) if and only if
there are positive integers $c$, $d$ such that $a^c=b^d$.

We now compute the dimensions of the terms of order 2 of the spectral sequence. The Betti numbers of
the Hopf surface are
\[
b_0=1,\quad b_1=1,\quad b_2=0,\quad b_3=1,\quad b_4=1.
\]
From \cite{Vai},  the basic Betti numbers are:
\[
e_0=1,\quad e_1=0,\quad e_2=1.
\]
By our lower estimate of $\dim E_2^{u,1}$ and the Poincar\'e duality, we obtain:

\begin{proposition}
The dimension of the terms of the second order of the canonical Vaisman
foliation associated to $H_{\alpha ,\beta }$ satisfy:
\begin{equation}
\begin{split} \label{ineq_spectr}
\dim E_2^{u,0}&=\dim E_2^{u,2}=\frac{1+(-1)^u}2,   \\
\dim E_2^{u,1}&\ge 1+(-1)^u.
\end{split}
\end{equation}
\end{proposition}

\begin{remark}
If $a^c=b^d$ for $c$, $d\in \mathbb{N}\backslash \{0\}$, then the foliation is quasi-regular
and we get equality in the second line  of \eqref{ineq_spectr}.
\end{remark}

\subsection{A suspension of an odd-dimensional torus} \label{ex 2} We use an idea presented in \cite{Kac-Nic2} to
construct foliations of arbitrary large transverse dimension.

Consider the matrix
\[
A:=\left(
\begin{array}{ccccc}
d_1 & 1 & 1 & \cdots & 1 \\
1 & d_2 & 0 & \cdots & 0 \\
1 & 0 & d_3 & \cdots & 0 \\
\vdots & \vdots & \vdots & \ddots & \vdots \\
1 & 0 & 0 & \cdots & d_{2n+1}
\end{array}
\right)
\]
where $d_1=1$ and $d_k=1+d_1\cdot d_2 \cdot \ldots \cdot d_{k-1}$, for $2\ge k\ge 2n+1$. Then:
\begin{enumerate}
\item $A\in \SL(2n+1,\ZZ)$.
\item All its  eigenvalues $\lambda _i$ have multiplicity 1 and are uniformly
distributed in the intervals $(d_1,d_3)$, $(d_3,d_4),\ldots,(d_{2n+1},\infty)$.
\item The coordinates $\{v_i^1,\ldots,v_i^{2n+1}\}$ of each eigenvector $v_i$ are li\-ne\-ar\-ly independent over the field $\mathbb{Q}$.
\end{enumerate}
 As in  \cite{Car} for dimension $2$
and \cite{Dom} for dimension $3$, we take the semidirect product $G:=\mathbb{R\rtimes R}^{2n+1}$
and use $A$ to induce a Lie group structure on $G$ by the
multiplication
\[
(t_1,z_1)\cdot (t_2,z_2):=(t_1+t_2,A^{t_1}z_2+z_1).
\]
The eigenvectors $v_i$ become tangent vectors at the origin of this Lie group, and we denote by $\{V_i\}_{1\le i\le 2n+1}$ the left invariant vector fields
induced by $\{v_i\}$. Consider the frame $\{V_0:=\partial /\partial
t,V_i\}$. The corresponding coframe determined by the canonical left invariant metric is
denoted by $\{\alpha _0:=dt,\alpha _i\}$, ${1\le i\le 2n+1}$.

As  $A\in \SL(2n+1,\mathbb{Z})$,
the subgroup $\Gamma :=\mathbb{Z\rtimes Z}^{2n+1}$ is
cocompact.
The quotient manifold $T_A^{2n+2}:=\Gamma \diagdown G$ is a suspension of the torus
$T^{2n+1}:=\mathbb{R}^{2n+1}\diagup \mathbb{Z}^{2n+1}$ with respect to the automorphism canonically induced on the
torus by the matrix $A$, \cite{Kac-Nic1, Car}. For the construction of a suspension we indicate \cite[Chapter I]{Mo}.

Notice that the metric, the left invariant frame and coframe can be
projected on $T_A^{2n+2}$; for convenience we use the same
symbols to denote the projected objects.

We construct a foliation which is different from the classical foliation associated to a suspension. Precisely, the fields $V_{2n}$, $V_{2n+1}$ generate two flows which  induce a foliation of leafwise dimension $2$ and transverse dimension $2n$ on
$T_A^{2n+2}$. As in \cite{Car, Dom}, the above metric is
bundle-like.

\begin{proposition}
For the above foliation the spectral terms $E_2^{2n}$ and $E_2^{0,1}$ vanish,
\[
E_2^{2n}\equiv 0, \quad E_2^{0,1}\equiv 0.
\]

\end{proposition}

\begin{proof} To compute $E_2^{2n}$ and $E_2^{0,1}$ we use an argument
similar to the one in \cite{Dom} (see also \cite{Car}).

We first look at $E_1^{u,v}$, $0\le u\le 2n$, $0\le v\le 2$. As $d_{0,1}\alpha _i=0$,
these terms decompose as
\begin{equation}
E_1^{u,v}=\bigwedge^u(\alpha _0,\ldots,\alpha _{2n-1})\otimes E_i^{0,v}.
\label{E_1_times}
\end{equation}
As $v^j_i$ are independent over $\Q$,   there exist some constants $C_i$ and $\delta _i$ such that the
following Diophantine condition is satisfied \cite[II.4.]{Sch}.
\[
\left| \left\langle m,v_i\right\rangle \right| \ge \frac{C_i}{\left\|
m\right\| ^{\delta _i}},
\]
for any $m\in \mathbb{Z}^{2n+1}\backslash \{0\}$ and any $1\le i\le 2n+1$.
Then, for any function
\[
h:\mathbb{R\times R}^{2n+1}\rightarrow \mathbb{R} \quad \text{with}\quad  h(t,z)=h(t+1,A(z))
\]
there exists a smooth function $f$ with the same properties such that
\[
V_i(f)=h-h_0,
\]
with $h_0\equiv h_0(t)=h_0(t+1)$. As in \cite{Dom}, using \eqref{E_1_times}
we obtain
\begin{equation}
E_1^{u,v}=\bigwedge^u(\alpha _0,\ldots,\alpha _{2n-1})\otimes
\bigwedge^v(\alpha _{2n,}\alpha _{2n+1})\otimes \Omega ^0(S^1).  \label{E_1}
\end{equation}
We use $E_2=H(E_1,d_1)$ to compute the dimension of $E_2^{2n,0}$.
Let $\alpha \in E_1^{2n,0}$. From \eqref{E_1} we can write
\[
\alpha =h(t)\alpha _0\wedge \cdots\wedge \alpha _{2n-1},
\]
with $h(t)=h(t+1)$. As in \cite[Proposition 2]{Car}, we construct a
differential form $\alpha ^{\prime }\in E_1^{2n-1,0}$ such that $d_1\alpha^{\prime }=\alpha $,
$d_1$ being  the basic de\thinspace Rham operator. We
choose
\[
\alpha ^{\prime }=f(t)\alpha _1\wedge \cdots\wedge \alpha _{2n-1}.
\]
Then we compute:
\begin{equation}
\begin{split} \nonumber
d_1\alpha ^{\prime } &=(f^{\prime }(t)-(\log \lambda _1+\cdots
+\log \lambda_{2n-1})f(t))\alpha _0\wedge\cdots\wedge \alpha _{2n-1} \\
&=(f^{\prime }(t)+\log \lambda f(t))\alpha _0\wedge\cdots\wedge \alpha _{2n-1},
\end{split}
\end{equation}
where $\lambda :=\lambda _{2n}\cdot \lambda _{2n+1}$. By the
distribution of the eigenvalues of $A$ on the real axis,  $\lambda
\not =1$, and the differential equation
\[
f^{\prime }(t)+\log \lambda f(t)=h(t)
\]
has the solution
$$f(t)=\lambda ^{-t}(k+\int_0^t\lambda ^xh(t)dx).$$
Choosing
$$k=\frac 1{\lambda -1}\int_0^1\lambda ^xh(t)dx,$$
 leads to
$f(t)=f(t+1)$. Then $E_2^{2n,0}\equiv 0$, and the foliation is taut.

Finally, we determine the kernel of the operator $d_1: E_2^{0,1}\rightarrow E_2^{1,1}$.
If $\alpha \in E_1^{0,1}$, then
\[
\alpha =f_1\alpha _{2n}+f_2\alpha _{2n+1},
\]
and $f_i\equiv f_i(t)=f_i(t+1)$. As above,
\begin{equation}
d_1\alpha =(f_1^{\prime }(t)-\log \lambda _{2n}f_1(t))\alpha _{2n}
+(f_2^{\prime }(t)-\log \lambda _{2n+1}f_2(t))\alpha _{2n+1}. \nonumber
\end{equation}
From $f_i^{\prime }(t)-\log \lambda _{2n-1+i}f_i(t)=0$, as the functions are
periodic, we get $f_i\equiv 0$, $1\le i\le 2$. Then $E_2^{0,1}\equiv 0$.
\end{proof}

\hfill

Using  \ref{criterii}, we obtain:

\begin{corollary}
The above foliation on $T^{2n+2}_A$ does not come from a Vaisman structure.
\end{corollary}

\begin{remark}
Any two vector fields from the set ${V_1,\ldots,V_{2n+1}}$ can be used to construct the foliation. However, if  $\frac{\partial}{\partial t},J(\frac{\partial}{\partial t}) \in T\caf$, where $J$ is the complex structure,
then the existence of a Vaisman structure on the above suspension of the torus would imply the existence of a Sasaki structure on the odd dimensional torus which is the fibre of the suspension. But it is known that odd-dimensional tori do not admit Sasaki structures (in fact, not even K-contact structures, see \cite{bg}, \cite{It}). This means that our obstruction becomes important only when {\em $\frac{\partial}{\partial t},J(\frac{\partial}{\partial t})$ are transverse to the foliation}.
\end{remark}

\begin{remark}
On the other hand $T_A^{2n+2}$ is orientable, and the differential
one-form $\omega _0$ $\equiv dt$ is closed but not exact, so $b_1>0$.
Then, the existence of a Vaisman structure on our manifold cannot be precluded using
this Betti number (see the Introduction). Consequently, in some instances,
the terms of the spectral sequence may be a useful obstruction in deciding wether a 2 dimensional
foliated structure on a closed manifold is generated by a Vaisman structure or not.
\end{remark}

\hfill
\hfill

{\noindent\bf Acknowledgment.}
We acknowledge useful discussions with J. \'Alvarez L\'opez concerning the geometrical meaning of the spectral term $E_2^{0,1}$.

\hfill

\hfill

{\small

\noindent {\sc Liviu Ornea\\
University of Bucharest, Faculty of Mathematics, \\14
Academiei str., 70109 Bucharest, Romania. \emph{and}\\
Institute of Mathematics ``Simion Stoilow" of the Romanian Academy,\\
21, Calea Grivitei Str.
010702-Bucharest, Romania }\\
\tt lornea@fmi.unibuc.ro, \ \ liviu.ornea@imar.ro

\hfill

\noindent{\sc Vladimir Slesar\\
Department of Mathematics, University of Craiova,\\
13 Al.I. Cuza Str., 200585-Craiova, Romania}\\
\tt slesar.vladimir@ucv.ro
}

\end{document}